\documentclass{article}
\usepackage[letterpaper, textwidth=5.0in]{geometry}

\usepackage{amssymb,amsbsy,amsmath,amsfonts,amssymb,amscd,amsthm}
\usepackage{graphicx}
\graphicspath{
  {images/}{../images/}
}
\usepackage{xcolor}
\usepackage{cancel}
\usepackage{hyperref}
\usepackage{ulem}
\usepackage{subcaption}
\usepackage{float}
\usepackage{comment}
\usepackage{authblk}

\newtheorem{dummy}{dummy}[section]

\newtheorem{proposition}[dummy]{Proposition}
\newtheorem*{Proposition*}{Proposition}
\newtheorem{lemma}[dummy]{Lemma}

\newtheorem{defi}[dummy]{Definition}
\newtheorem{Remark}[dummy]{Remark}

\title{
On the Gram matrix of standard inner products of asymmetrically-weighted Hermite functions
}
\author{
Ruiyang Dai
\thanks{%
\url{ruiyang.dai@inria.fr}}
}
\affil{\small
Project-Team Makutu, Inria, \\
University of Pau and Pays de l'Adour, \\
Pau, France.}
\date{}

\begin{document}
\maketitle
\begin{abstract}
Let \( A = (a_{n,m})_{n,m\geq0} \) denote the infinite Gram matrix 
associated with the standard \(L^2\) inner product of 
asymmetrically-weighted (AW) Hermite functions. 
We derive an explicit representation of its entries and its Cholesky factorization.
We further show that this factorization 
admits a natural interpretation 
on a scaled Bargmann--Fock basis. 
An explicit formula for the inverse of \(A\) is also obtained.
We then consider the corresponding finite Gram matrix
and analyze its asymptotic property, 
as well as that of its Schur complement. 
The analysis is motivated by numerical methods for plasma physics, 
in particular Galerkin spectral methods 
applied to the Vlasov--Poisson (VP) system. 
As an application, we demonstrate how the derived Gram matrix formulas 
and asymptotic results can be exploited in the analysis 
and implementation of a Galerkin spectral method for the VP system.
\end{abstract}

\tableofcontents
\section{Introduction}

Orthogonal polynomials serve as basis functions 
for approximation functions in a wide range of theoretical and numerical applications, 
see for example \cite{nist,szego,askey1975orthogonal}.
One popular family of 
orthogonal polynomials is the Hermite polynomials 
$(H_m)_{m\in\mathbb{N}}$. 
Throughout this paper, we adopt the convention that $0\in\mathbb{N}$.
The Hermite polynomials are orthogonal on $(-\infty, \infty)$ 
with respect to the weight  function $\exp(-v^2)$, 
equivalently,
\begin{equation}\label{eq:hermite_polynomials_orthogonality}
  \int_{\mathbb{R}} H_n(v) H_m(v) e^{-v^2} \,dv
= \sqrt{\pi} 2^m m! \,\delta_{n,m}, 
\quad n,m \in \mathbb{N}.
\end{equation}
In this paper, we are interested in 
the scaled AW Hermite functions,
which are widely used for the numerical discretization in physics.
In particular, 
an appropriate scaling in the velocity variable 
can substantially reduce the number of Hermite modes 
required to accurately represent functions with approximately Gaussian profiles, 
making the resulting spectral method more efficient 
\cite{joyce1971numerical,gottlieb1977numerical,tang1993hermite,holloway1996spectral,schumer1998vlasov,gagne1977splitting}.
Hermite functions combine Hermite polynomials with a Gaussian function 
to form a suitable set of basis functions. 
This choice is particularly appropriate for many plasma physics problems, 
where the solutions exhibit exponential decay as
$|v| \to \infty$.
We define the scaled AW Hermite functions by
\begin{equation}\label{eq:scaled_aw_hermite_functions}
  \psi_n(v)
:=\psi_{n,T}(v)
= \left( {2^n n! \sqrt{\pi T}} \right)^{-\frac12}
  {H}_n(\frac{v}{\sqrt{T}}) \sigma(v),
\quad n \geq 0.
\end{equation}
where the Gaussian function 
\[
\sigma(v) = \exp(-\dfrac{v^2}{T}),
\]
and $T>0$ is a scaling parameter.
The scaled AW Hermite functions satisfy the orthogonality
\begin{equation}\label{eq:scaled_aw_hermite_functions_orthogonality}
  \int_{\mathbb{R}} \psi_n(v) \psi_m(v) \omega(v) \,dv
= \delta_{n,m}, 
\quad n,m \in \mathbb{N},
\end{equation}
where the weight function 
\[
\omega(v) = \exp(\dfrac{v^2}{T}).
\]
The natural inner product to be considered 
is the weighted $L^2$ inner product
\begin{equation}\label{eq:L2_weighted_inner_product}
  \left(f, g \right)_{L^2(\omega(v)dv)} 
= \int_{\mathbb{R}} f(v) g(v) \omega(v) \,dv.
\end{equation}
We also consider the standard $L^2$ inner product 
of the scaled AW Hermite functions on $\mathbb{R}$.
The standard $L^2$ inner product of two scaled AW Hermite functions 
$f$ and $g$ is defined as 
\begin{equation}\label{eq:L2_inner_product}
\left(f, g \right)_{L^2(dv)} = \int_{\mathbb{R}} f(v) g(v) \,dv.
\end{equation}
The infinite Gram matrix $A$ is the collection of 
all scalar products of the scaled AW Hermite functions.
We take the first $N+1$ scaled AW Hermite functions and 
build the $(N+1)\times(N+1)$ finite Gram matrix $A_{11}^N$ with 
entries 
\[
   (A_{11}^N)_{n,m} 
:= \left(\psi_n, \psi_m \right)_{L^2(dv)},
   \quad 0 \leq n,m \leq N.
\]
The goal of this article is to analyze the infinite Gram matrix $A$ 
and its block $A_{11}^N$ which is finite-dimensional.
In particular, we are interested in the asymptotic properties of $A_{11}^N$ 
and its applicability in numerical methods, given that $A_{11}^N$ is dense 
and ill-conditioned.

The motivation for this work arises from 
the numerical simulation of plasma physics problems, 
with particular emphasis on the numerical stability of discretization schemes 
and their rigorous convergence analysis. 
Nevertheless, the results presented here are independent of 
this specific motivation and may be applicable more broadly. 
Readers interested primarily in the results, 
rather than their original motivation, 
may therefore skip the following paragraph.

The simplest kinetic model describing collisionless plasmas is the VP system. 
A wide range of numerical methods has been developed for its approximation, 
which can be broadly classified into two categories: particle-in-cell (PIC) methods \cite{birdsall2018plasma} and Eulerian methods \cite{filbet2003comparison}.
In this paper, we focus on Eulerian methods based on polynomial bases, 
leading to spectral discretizations. 
In particular, a widely used method for the numerical approximation 
of the VP system is the Petrov--Galerkin spectral method. 
In such method, the distribution function is represented 
by a finite set of orthogonal polynomials in velocity space, 
rather than by a direct discretization of the velocity variable.
This idea dates back to the pioneering works of the 1960s \cite{armstrong1966numerical,joyce1971numerical}.
Among the various polynomial bases, 
Hermite polynomials have been the most extensively employed 
in Petrov--Galerkin formulations of the VP system. 
In particular, the Petrov--Galerkin method based on AW Hermite functions has become a standard spectral approach. 
However, this formulation yields a system that is known to suffer from numerical instability 
\cite{holloway1996spectral,kormann2021generalized,bessemoulin2022stability,bessemoulin2023convergence,manzini2017convergence,funaro2021stability,pagliantini2023physics}.
Moreover, no estimate is currently available for 
the norm of the solution to the VP system 
in the corresponding scaled, time-independent weighted Sobolev space.
To address this issue, \cite{dai2024quadratic,dai2026galerkin} 
proposed a Galerkin method based on AW Hermite functions, 
since weighted $L^2(\mathbb{R})$ space 
is embedded in $L^2(\mathbb{R})$ space for certain weight functions.
More recently, \cite{dai2026convergence}
established the convergence of Galerkin approximations 
based on AW Hermite functions for the VP system, 
proving error estimates and spectral accuracy in Sobolev spaces.
While this approach improves the stability properties of the resulting scheme, 
the AW Hermite functions are no longer orthogonal under the Galerkin inner product, 
leading to the appearance of a Gram matrix in the discrete system.

The Gram matrix considered in this paper 
may be viewed as a Hermite-type case of a moment matrix. 
The study of moment matrices has a long history, 
and the asymptotic behavior of their smallest eigenvalues 
has been investigated \cite{szego1936some,beckermann2000condition,berg2002small,chen1999small,chen2004smallest,widom1966small,wilf2012finite}. 
For example, in \cite{szego1936some}, 
Szeg\"o considers the cases of infinite intervals with the weights
$\omega(x) = \exp(-x^2), \, x\in (-\infty, +\infty)$ and
$\omega(x) = \exp(-x),   \, x\in [0, +\infty)$,
which correspond to the Hermite and Laguerre polynomials, respectively.
A standard approach to studying the asymptotic behavior of 
the smallest eigenvalue is to estimate its reciprocal, 
which can be characterized by a Rayleigh quotient. 
We adopt this approach here. 
In particular, since all the entries of \((A_{11}^N)^{-1}\) are positive, 
its Rayleigh quotient can be estimated explicitly.

The remainder of the paper is organized as follows.
In Section~\ref{sec:hermite-functions}, 
we introduce the necessary notation and recall the definition and 
fundamental properties of the scaled AW Hermite functions. 
In Section~\ref{sec:infinite-gram-matrix}, 
we derive an explicit representation 
of the entries of the Gram matrix associated with 
the standard $L^2$ inner product and its Cholesky decomposition. 
We further show that the factorization of the infinite Gram matrix 
can be interpreted on a scaled Bargmann–Fock basis.
An explicit expression for the inverse of 
the infinite Gram matrix is also provided. 
In Section~\ref{sec:finite-gram-matrix}, 
we analyze the corresponding finite Gram matrix.
Section~\ref{sec:decay-properties-of-the-finite-gram-matrix} 
investigates the decay properties of 
the finite Gram matrix and those of the Schur complement.
Finally, Section~\ref{sec:conclusion} concludes the paper. 
In Appendix~\ref{sec:application-to-the-vp-system}, 
we illustrate the applicability of the results to the VP system.

\section{Hermite functions}
\label{sec:hermite-functions}

We denote by $L^2_\omega(\mathbb{R})$ the weighted Lebesgue space of 
measurable functions $f:\mathbb{R}\to\mathbb{R}$ such that
\[
   \| f \|_{L_\omega^2}^2
:= \int_{\mathbb{R}} f^2(v) \omega(v) \,dv
 = \int_{\mathbb{R}} f^2(v) \exp\left(\frac{v^2}{T}\right) \,dv
 < \infty.
\]
The associated weighted inner product is~\eqref{eq:L2_weighted_inner_product}.
The scaled AW Hermite functions~\eqref{eq:scaled_aw_hermite_functions} 
can also be defined by the recursive scheme
\begin{equation}\label{recursive_relation_v}
\begin{aligned}
& \psi_0(v) = \left( \pi T \right)^{-\frac14} \sigma(v), \quad \\
& v\psi_m(v) 
= \sqrt{\dfrac{T}{2}} 
  \left( \sqrt{m} \psi_{m-1} + \sqrt{m+1} \psi_{m+1} \right),
  \quad m = 1,2,\dots.
\end{aligned}
\end{equation}
We also have the following recursive relation
\begin{equation}\label{recursive_relation_partial_v}
   \dfrac{\partial \psi_m(v)}{\partial v} 
= -\sqrt{\dfrac{2}{T}} \sqrt{m+1} \psi_{m+1}
   \quad m = 0,1,2,\dots,
\end{equation}
The set of functions $\psi_n(v)$ is 
the $L^2_\omega(\mathbb{R})$--orthogonal system, 
namely,~\eqref{eq:scaled_aw_hermite_functions_orthogonality}.

Throughout this paper, we consider the usual Lebesgue space $L^2(\mathbb{R})$. 
Since the weight function
\[
\omega(v)=\exp\left(\frac{v^2}{T}\right)\ge 1,
\quad \forall v \in \mathbb{R},
\]
it follows that for every $f\in L^2_\omega(\mathbb{R})$,
\[
    \int_{\mathbb{R}} f^2(v) \,dv
\le \int_{\mathbb{R}} f^2(v) \omega(v) \,dv
<   \infty,
\]
which shows that
\[
L^2_\omega(\mathbb{R})\subset L^2(\mathbb{R})
\]
with continuous embedding.
We consider the inner product~\eqref{eq:L2_inner_product}.
We can see that $\psi_m(v)$ is even function if $m$ is even, 
and $\psi_m(v)$ is odd function if $m$ is odd. 
Therefore the integral of $\psi_m(v)\psi_n(v)$ over $\mathbb{R}$ vanishes, that is 
\[
\int_{\mathbb{R}} \psi_m(v)\psi_n(v) \,dv = 0, 
\quad m+n \text{ is odd}.
\] 

\section{Infinite Gram matrix}
\label{sec:infinite-gram-matrix}

\subsection{Coefficients of the infinite Gram matrix}
\begin{defi}
The infinite symmetric Gram matrix
$
A=A^T=(a_{mn})_{m,n\geq 0}\in \mathbb R^{\mathbb N\times \mathbb N}
$
is the collection of all scalar products of 
the scaled AW Hermite functions,
where the coefficients are
\[
a_{mn} = a_{nm} = \int_{\mathbb{R}} \psi_m(v) \psi_n(v) \,dv.
\]
\end{defi}

The coefficients of the matrix $A$ are $L^2$ 
scalar products of AW Hermite functions. 
These coefficients are computable in finite terms 
since the product of two AW Hermite functions 
can be expressed as a Gaussian function multiplied by a polynomial function.
However, to our knowledge,  the exact value of these coefficients 
is not available in the literature on special functions \cite{nist,szego,magnus}.
In this Section, we propose some formulas for the calculation of 
the quadratic scalar product of AW Hermite functions.

\begin{proposition}\label{prop:4.1}
If the sum of the indices $m,n$ ($m,n\ge 0$) 
is odd, then $a_{mn}=0$. 
Otherwise 
\begin{equation} \label{eq:b46}
  a_{mn}
= (-1)^{\frac{m-n}2} 
  2^{-(m+n) - \frac{1}2} 
  \frac{(m+n)!}{ \left( \frac{m+n}2 \right)! \sqrt{m!n!} }.
\end{equation}
\end{proposition}
\begin{proof}
If $m+n$ is odd, then $\psi_m \psi_n$ is equal to a Gaussian function multiplied by an odd polynomial, so its integral vanishes. In this case $a_{mn}=0$.
If $m+n$ is even, we have 
\[
  \sqrt{ \frac{2m}{T} } \  a_{mn}
= \sqrt{ \frac{2m}{T} } \ \int_{\mathbb{R}} \psi_m(v) \psi_n(v)dv.
\]
Using the general identity~\eqref{recursive_relation_partial_v}, 
we can write
\[
\begin{aligned}
   \sqrt{ \frac{2m}{T} } \  a_{mn}
&=-\int_{\mathbb{R}} \psi_{m-1}'(v) \psi_n (v) dv
 = \int_{\mathbb{R}} \psi_{m-1}(v)  \psi_n'(v) dv \\
&=-\sqrt{ \frac{2(n+1)}{T} } \int_{\mathbb{R}} \psi_{m-1}(v) \psi_{n+1}(v)dv \\
&=-\sqrt{ \frac{2(n+1)}{T} } \  a_{m-1,n+1}.
\end{aligned}
\]
That is
\begin{equation} \label{eq:b40-rab}
\sqrt{ m } \  a_{mn}=- \sqrt{ n+1 }   \  a_{m-1,n+1}.
\end{equation}
We get by iteration
$$
  {\left( m(m-1) \dots 2 \right) }^\frac12 a_{mn}
= (-1)^m { \left( n(n+1) \dots (m+n)\right)}^\frac12 a_{0,m+n},
$$
that is
\begin{equation} \label{eq:b41}
  a_{mn}
= (-1)^m \left( {\frac{(n+m)!}{n!m!}} \right)^\frac12 
  a_{0,m+n}.
\end{equation}
The technical Lemma \ref{lem:4.2} yields the value of $a_{0,m+n}$ 
from which we obtain 
$$
  a_{mn}
= (-1)^m \left( {\frac{(n+m)!}{n!m!}} \right)^\frac12
  (-1)^{(m+n)/2} 
  2^{-(m+n) - \frac{1}2}  
  \frac{ (m+n)!^\frac12} { \left( \frac{m+n}2 \right)! },
$$
that is
$$
  a_{mn}
= (-1)^{\frac{m-n}2} 
  2^{-(m+n) - \frac{1}2} 
  \frac{(m+n)!}{ \left( \frac{m+n}2 \right)! \sqrt{m!n!} }.
$$
\end{proof}
\begin{lemma} \label{lem:4.2}
If $m$ is even, we have
\[
  a_{0m}
= (-1)^{\frac{m}{2}} 2^{-m-\frac12}  
 \frac{(m!)^\frac12}{(\frac{m}{2})!}.
\]
\end{lemma}
\begin{proof}
We have
$
  \psi_0(v)\psi_m(v)
= (\pi T)^{-\frac12} (2^mm!)^{-\frac12} 
  e^{-2v^2/T}H_m\left( \frac{v}{\sqrt{T}} \right)
$. 
To be able to  perform a rescaling in this expression, 
we can use  the general formula \cite[page 255]{magnus}
$$
  H_m(\lambda x)
= \sum_{l=0}^{[\frac{m}{2}]} \lambda^{m-2l} (\lambda^{2}-1)^l  
  \frac{m!}{(m-2l))!l!}H_{m-2l}(x).
$$ 
Take $\lambda = \frac{1}{\sqrt 2}$ 
and  $x= \frac{\sqrt 2 v}{\sqrt{T}}$. 
Then
$$
  H_m\left(\frac{v}{\sqrt{T}}\right)
= \left( -  \frac{1}{2} \right)^{\frac{m}{2}}
  \frac{m!}{(\frac{m}{2})!} + R(v),
$$
where the residual $R(v)$ is orthogonal 
to the weight $e^{-\frac{2v^2}{T}}$ 
because it is a linear combination of 
Hermite polynomials of degree $\geq 1$ (with convenient weight).
We obtain
$$
  a_{0m}
= \int_{\mathbb{R}} \psi_0(v)\psi_m(v) dv 
= (\pi T)^{-\frac12} (2^mm!)^{-\frac12} 
  \left( -\frac{1}{2} \right)^{\frac{m}{2}}  
  \frac{m!}{(\frac{m}{2})!} \sqrt{\frac{T\pi}{2}},
$$
which yields the claim after simplification.
\end{proof}

Unfortunately the formula~\eqref{eq:b46} cannot be  used to calculate the coefficients 
in a stable manner because the calculation 
on the computer of the factorial of large natural numbers is difficult.
We have the following formulas which provide a stable method to calculate all coefficients.
\begin{lemma} \label{lemma:nm}
The coefficients of the Gram matrix can be evaluated with computationally stable formulas.
\\
i) To calculate the diagonal coefficients of the Gram matrix, use the recurrence formulas
\begin{equation} \label{eq:iter1}
\left\{
\begin{aligned}
& a_{00}
= 2^{-\frac12},\\
&  a_{m+1,m+1}
= \frac{2m+1}{2m+2}a_{mm}, \quad m\geq 0.
\end{aligned}
\right.
\end{equation}
ii) To calculate the upper extra-diagonal coefficients of the Gram matrix, 
use the recurrence formulas which start from the diagonal
\begin{equation} \label{eq:iter2}
  \forall m\geq 1: \quad 
  a_{m-l-1, m+l+1}
=-\sqrt{\frac{m-l}{m+l+1}} a_{m-l, m+l}, \quad 
  l=0, \dots, m-1,
\end{equation}
iii) The lower diagonal coefficients are equal to the  upper extra-diagonal coefficients
\begin{equation} \label{eq:iter3}
  \forall m\geq 1: \quad 
  a_{m+l, m-l}
= a_{m-l, m+l} \quad 
  \mbox{ for } 1 \leq l \leq  m.
\end{equation}
\end{lemma}
\begin{proof}
The first  set~\eqref{eq:iter1} of formulas are deduced from~\eqref{eq:b46}. 
The second set~\eqref{eq:iter2} of formulas are deduced from~\eqref{eq:b46}. 
The Gram matrix being symmetric, the symmetry~\eqref{eq:iter3} is trivial.
The computational stability of the formulas is because the only operations are multiplication by positive  numbers $\leq 1$.
\end{proof}

\subsection{Cholesky decomposition of the infinite Gram matrix}
\begin{proposition}[Cholesky decomposition of the Gram matrix]
\label{prop:Cholesky_decomposition_of_the_Gram_matrix}
The Gram matrix $A$ admits a unique Cholesky decomposition of the form
\[
A = L D L^\top,
\]
where \(L\) is a unit lower triangular matrix whose entries are given by
\begin{equation}\label{eq:entries_cholesky_gram_L}
\left\{
\begin{aligned}
&  l_{nm} = 
   (-1)^{\frac{n-m}{2}}
   \sqrt{\dfrac{n!}{m!}} 
   \dfrac{1}{4^{\frac{n-m}2}(\frac{n-m}2)!},
&& n-m \in 2\mathbb{N},  \\
&  l_{nm} = 0,
&& \text{otherwise}.\\
\end{aligned}
\right.
\end{equation}
and $D$ is a diagonal matrix given by 
\begin{equation}\label{eq:entries_cholesky_gram_D}
D = \operatorname{diag}( 2^{-m-\frac{1}{2}})_{m\geq 0}.
\end{equation}
\end{proposition}
\begin{proof}
Since the Gram matrix $A$ is real, symmetric, and positive definite, 
and its diagonal entries are strictly positive,
it admits a unique Cholesky decomposition.
By the identity~\cite[18.5.13]{nist}, 
the scaled AW Hermite functions defined in~\eqref{eq:scaled_aw_hermite_functions} can be written as 
\[
  \psi_n(v)
= \sum_{k=0}^n c_{nk} 
  \left( \dfrac{v}{\sqrt T} \right)^k e^{-\frac{v^2}{T}},
\]
with
\[
\left\{
\begin{aligned}
&  c_{nk} = 
   (\pi T)^{-\frac14}
   (-1)^{\frac{n-k}{2}}
   \dfrac{\sqrt{n!}}{2^{\frac{n}2-k}(\frac{n-k}2)! k!},
&& n-k \in 2\mathbb{N},  \\
&  c_{nk} = 0,
&& \text{otherwise}.\\
\end{aligned}
\right.
\]
Hence, the infinite Gram matrix $A$ admits the factorization
\[
A = C G C^\top,
\]
where $C = (c_{nk})_{n,k\ge  0}$  and $G=(g_{kl})_{k,l\ge 0}$ is the moment matrix defined by
\[
  g_{kl} 
= \int_{\mathbb{R}} 
  \left( \dfrac{v}{\sqrt T} \right)^{k+l}
  e^{-\frac{2 v^2}{T}} \,dv
  \quad k,l \in \mathbb{N}.
\]
By the orthogonality of the Hermite polynomials~\eqref{eq:hermite_polynomials_orthogonality} written as scaled 
\[
  \int_{\mathbb{R}} 
  H_n\left( \sqrt{\frac2T} v\right) 
  H_m\left( \sqrt{\frac2T} v\right) 
  e^{-\frac{2 v^2}{T}} \,dv
= \sqrt{\dfrac{\pi T}{2}} 2^n n! \,\delta_{n,m}, 
\quad n,m \in \mathbb{N}.
\]
Let
\[
  q_n(v) 
:=\left( \sqrt{\dfrac{\pi T}{2}} 2^n n! \right)^{-\frac12} 
  H_n\left( \sqrt{\frac2T} v\right), 
  \quad n \in \mathbb{N}.
\]
Using the connection formulas for the Hermite polynomials~\cite[18.18.20]{nist}
\[
  (\sqrt2 x)^n 
= \dfrac{n!}{2^n}
  \sum_{m=0}^{\lfloor \frac{n}2 \rfloor}
  \dfrac{1}{m!(n-2m)!} H_{n-2m}(\sqrt2 x),
  \quad x \in \mathbb{R}.
\]
we obtain
\[
  \left( \dfrac{v}{\sqrt T} \right)^n 
= \left( \dfrac{\pi T}2 \right)^{\frac14} 
  \sum_{\substack{0\le k\le n \\ k\equiv n \,(\mathrm{mod} 2) }}
  \dfrac{n!}{2^{\frac{3n-k}{2}} \sqrt{k!} \left(\frac{n-k}{2}\right)! } q_k(v),
  \quad n \in \mathbb{N}.
\]
Hence, the infinite Gram matrix $A$ admits the following factorization
\[
A = C B (C B)^\top,
\]
where the matrix $B=(b_{nk})_{n,k\ge 0}$ is defined by
\[
\left\{
\begin{aligned}
&  b_{nk} = 
   \left( \dfrac{\pi T}2 \right)^{\frac14}
   \dfrac{n!}{2^{\frac{3n-k}2} \sqrt{k!} (\frac{n-k}2)!},
&& n-k \in 2\mathbb{N},  \\
&  b_{nk} = 0,
&& \text{otherwise}.\\
\end{aligned}
\right.
\]
If $n$ and $m$ have different parity, we have
\[
\sum_{k=0}^{\infty} c_{nk} b_{km} = 0.
\]
If $n$ and $m$ have the same parity, we have
\[
\begin{aligned}
 &\sum_{k=0}^{\infty} c_{nk} b_{km} \\
=&\sum_{k=m}^{n} c_{nk} b_{km} \\
=&\sum_{\substack{m\le k\le n \\ k\equiv n \,(\mathrm{mod} 2) }}
  (\pi T)^{-\frac14}
  (-1)^{\frac{n-k}{2}}
  \dfrac{\sqrt{n!}}{2^{\frac{n}2-k}(\frac{n-k}2)! k!} \,
  \left( \dfrac{\pi T}2 \right)^{\frac14}
  \dfrac{k!}{2^{\frac{3k-m}2} \sqrt{m!} (\frac{k-m}2)!} \\
=&2^{-\frac14}
  \sqrt{\dfrac{n!}{m!}}
  \dfrac{1}{2^{\frac{n-m}2}}
  \sum_{\substack{m\le k\le n \\ k\equiv n \,(\mathrm{mod} 2) }}
  (-1)^{\frac{n-k}{2}}
  \dfrac{2^{-\frac{k}2}}{ (\frac{n-k}2)! (\frac{k-m}2)!} \\
=&2^{-\frac14}
  \sqrt{\dfrac{n!}{m!}}
  \dfrac{1}{2^{\frac{n-m}2}} \,
  \left(
  (-1)^{\frac{n-m}{2}}
  \dfrac{2^{-\frac{n}2}}{ (\frac{n-m}2)! } 
  \right) \\
=&2^{-\frac{m}2-\frac14} \,
  (-1)^{\frac{n-m}{2}}
  \sqrt{\dfrac{n!}{m!}} 
  \dfrac{1}{4^{\frac{n-m}2}(\frac{n-m}2)!}.
\end{aligned}
\]
Hence, we obtain
\[
A = (C B) (C B)^\top
  = L D L^\top,
\]
where the entries of the infinite matrices $L$ and $D$ are 
given by~\eqref{eq:entries_cholesky_gram_L} 
and~\eqref{eq:entries_cholesky_gram_D}, respectively. 
\end{proof}
\begin{Remark}
Expanding the product \(LDL^\top\) entrywise 
and using the fact that its entries are equal to 
those of the infinite Gram matrix \(A\)~\eqref{prop:4.1} yields the identity
\begin{equation}\label{eq:identity_hypergeometric_series}
  \sum_{\substack{k=0 \\ k\equiv n \,(\mathrm{mod} 2) }}^{\operatorname{min}(n,m)}
  \dfrac{m!n!}{k!} 
  \dfrac{1}{2^{-k} (\frac{n-k}{2})! (\frac{m-k}{2})!}
= \dfrac{(m+n)!}{(\frac{m+n}{2})!},
\end{equation}
where $m,n$ are nonnegative integers such that $m+n$ is even.

This identity also admits a direct proof.
Indeed, after rewriting the left-hand side in terms of Pochhammer symbols,
the resulting expression can be recognized as a hypergeometric series. 
The identity then follows immediately from the special case of the hypergeometric function
\[
{}_2F_1(-m,b;c;1) = \dfrac{(c-b)_m}{(c)_m}.
\]

Thus, the Cholesky decomposition of the infinite Gram matrix $A$ 
can also be proved by expanding the product $L D L^\top$ explicitly 
and using the identity~\eqref{eq:identity_hypergeometric_series}.
\end{Remark}
\begin{lemma}\label{lem:hypergeometric_identity}
Let $m,n$ be nonnegative integers such that $m+n$ is even. Then
\begin{equation}\label{eq:identity_hypergeometric_series}
\sum_{\substack{k=0\\ k\equiv n\,(\mathrm{mod}\,2)}}^{\min(m,n)}
\frac{m!\,n!}{k!}
\frac{2^k}
{\left(\frac{n-k}{2}\right)!
 \left(\frac{m-k}{2}\right)!}
=
\frac{(m+n)!}
{\left(\frac{m+n}{2}\right)!}.
\end{equation}
\end{lemma}

\subsection{Bargmann--Fock space interpretation}

Let $\mathcal{F}^2(\mathbb{C})$ denote the classical Bargmann--Fock
space
\[
\mathcal{F}^2(\mathbb{C})
=
\left\{
F\in\operatorname{Hol}(\mathbb{C}) :
\|F\|_{\mathcal F}^2
=
\frac1\pi
\int_{\mathbb C}|F(z)|^2e^{-|z|^2}\,dA(z)<\infty
\right\},
\]
equipped with the inner product
\[
(F,G)_{\mathcal F}
=
\frac1\pi
\int_{\mathbb C}
F(z)\overline{G(z)}e^{-|z|^2}\,dA(z).
\]
For $n\in\mathbb N$, define
\[
g_n(z)
=
\exp\left(-\frac14\partial_z^2\right)
\frac{z^n}{\sqrt{n!}}.
\]
Since the exponential differential operator acts finitely on every
polynomial, we have 
\[
  g_n(z)
= \sum_{k=0}^{\lfloor n/2\rfloor}
  \frac{\sqrt{n!}}
  {k!\sqrt{(n-2k)!}}
  \left(-\frac14\right)^k
  \frac{z^{n-2k}}{\sqrt{(n-2k)!}}.
\]
Define the $T$-dependent Bargmann transform by
\[
(\mathcal B_Tf)(z)
=
\frac{\sqrt2}{(\pi T)^{1/4}}
\int_{\mathbb R}
\exp\left(
-z^2+2\sqrt{\frac2T}\,zv-\frac{v^2}{T}
\right)f(v)\,dv,
\]
and introduce the rescaled transform
\[
  (\mathcal S_Tf)(z)
:=(\mathcal B_Tf)\left(\frac z{\sqrt2}\right).
\]
\begin{proposition}\label{prop:bargmann-transform}
For every $n\in\mathbb N$,
\[
\mathcal B_T\psi_{n,T}=g_n.
\]
Moreover, 
$\mathcal S_T$ is injective from
$L^2(\mathbb R)$ onto $\mathcal F^2(\mathbb C)$ and satisfies
\[
(\mathcal S_Tf,\mathcal S_Th)_{\mathcal F}
=
\sqrt2\,(f,h)_{L^2(\mathbb R)},
\qquad f,h\in L^2(\mathbb R).
\]
Consequently, $2^{-1/4}\mathcal S_T$ is a unitary operator from
$L^2(\mathbb R)$ onto $\mathcal F^2(\mathbb C)$.
\end{proposition}
\begin{proof}
Consider the generating function
\[
\Psi_T(v,t)
=
\sum_{n=0}^{\infty}
\psi_{n,T}(v)\frac{t^n}{\sqrt{n!}}.
\]
Using the generating function of the Hermite polynomials,
\[
  \sum_{n=0}^{\infty} H_n(x)\frac{u^n}{n!}
= e^{2xu-u^2}, 
\quad \text{ with }
x = \frac{v}{\sqrt T},\, 
u = \frac{t}{\sqrt 2},
\]
we obtain
\[
\Psi_T(v,t)
=
(\pi T)^{-1/4}
\exp\left(
-\frac{v^2}{T}
+\sqrt{\frac2T}\,tv
-\frac{t^2}{2}
\right).
\]
Therefore,
\[
\begin{aligned}
  (\mathcal B_T\Psi_T(\cdot,t))(z)
&=\frac{\sqrt2}{\sqrt{\pi T}}
  e^{-z^2-t^2/2}
  \int_{\mathbb R}
  \exp\left(
  -\frac{2v^2}{T}
  +\sqrt{\frac2T}(2z+t)v
  \right)\,dv \\
&=e^{-z^2-t^2/2}
  \exp\left(\frac{(2z+t)^2}{4}\right) \\
&=\exp\left(tz-\frac{t^2}{4}\right).
\end{aligned}
\]
On the other hand,
\[
  \sum_{n=0}^{\infty} g_n(z)\frac{t^n}{\sqrt{n!}}
= \exp\left(-\frac14\partial_z^2\right)e^{tz} 
= \exp\left(tz-\frac{t^2}{4}\right).
\]
Comparison of the coefficients of $t^n/\sqrt{n!}$ yields
\[
\mathcal B_T\psi_{n,T}=g_n.
\]
It remains to establish the norm identity. Let $\mathcal B_0$ denote
the standard unitary Bargmann transform,
\[
(\mathcal B_0q)(z)
=
\pi^{-1/4}
\int_{\mathbb R}
e^{-y^2/2+\sqrt2yz-z^2/2}q(y)\,dy.
\]
In the integral defining $\mathcal S_T$, performing the substitution
\(
v=\sqrt{\frac T2}\,y,
\)
we obtain
\[
  (\mathcal S_Tf)(z)
= (\mathcal B_0q)(z),
\]
where
\[
  q(y)
= T^{1/4} f\left(\sqrt{\frac T2}\,y\right).
\]
A change of variables gives
\[
\begin{aligned}
(q_1,q_2)_{L^2(\mathbb R)}
&=
T^{1/2}
\int_{\mathbb R}
f_1\left(\sqrt{\frac T2}\,y\right)
f_2\left(\sqrt{\frac T2}\,y\right)
\,dy \\
&=
\sqrt2\,
(f_1,f_2)_{L^2(\mathbb R)}.
\end{aligned}
\]
The unitarity of $\mathcal B_0$ therefore implies
\[
  (\mathcal S_T f_1, \mathcal S_T f_2)_{\mathcal F}
= (\mathcal B_0 q_1, \mathcal B_0 q_2)_{\mathcal F} 
= (q_1,q_2)_{L^2(\mathbb R)} 
= \sqrt2\,(f_1, f_2)_{L^2(\mathbb R)}.
\]
Consequently,
\[
\left\|2^{-1/4}\mathcal S_Tf\right\|_{\mathcal F}
=
\|f\|_{L^2(\mathbb R)},
\]
and the surjectivity of $\mathcal B_0$, together with the invertibility
of the dilation
\[
f(v)\longmapsto
T^{1/4}f\left(\sqrt{\frac T2}\,v\right),
\]
shows that $2^{-1/4}\mathcal S_T$ is unitary.
\end{proof}
We now consider the infinite Gram matrix $A$. 
By Proposition~\ref{prop:bargmann-transform}, 
we have 
\[
\left(
g_m\left(\frac{\cdot}{\sqrt2}\right),
g_n\left(\frac{\cdot}{\sqrt2}\right)
\right)_{\mathcal F}
=
\sqrt2\,
(\psi_{m,T},\psi_{n,T})_{L^2(\mathbb R)},
\qquad m,n\in\mathbb N.
\]
It follows that
\[
A
=
\frac1{\sqrt2}
\left(
\left(
g_m\left(\frac{\cdot}{\sqrt2}\right),
g_n\left(\frac{\cdot}{\sqrt2}\right)
\right)_{\mathcal F}
\right)_{m,n\geq0}.
\]
To express this Gram matrix in terms of the normalized monomial basis,
write
\[
  g_n\left(\frac z{\sqrt2}\right)
= \sum_{j=0}^{n}
  \ell_{nj}
  \frac{(z/\sqrt2)^j}{\sqrt{j!}},
\]
Since the normalized monomials form an orthonormal basis of
$\mathcal F^2(\mathbb C)$,
\[
\left(
\frac{z^j}{\sqrt{j!}},
\frac{z^k}{\sqrt{k!}}
\right)_{\mathcal F}
=
\delta_{jk}.
\]
Consequently,
the infinite Gram matrix therefore admits the factorization
\[
  A 
= 
  LDL^*.
\]
This identity is initially understood as an identity of sesquilinear
forms on the space of finitely supported sequences.

\subsection{Inverse of the infinite Gram matrix}

The matrix \(L\) represents, on the normalized monomial basis,
the action of
\(
\exp\left(-\frac14\partial_z^2\right).
\)
Its inverse therefore represents
\(
\exp\left(\frac14\partial_z^2\right).
\)
Indeed,
\[
\begin{aligned}
  \exp\left(\frac14\partial_z^2\right)
  \frac{z^n}{\sqrt{n!}}
&=\sum_{k=0}^{\lfloor n/2\rfloor}
  \frac{1}{4^k k!}\,
  \partial_z^{2 k}\frac{z^n}{\sqrt{n!}} \\
&=\sum_{k=0}^{\lfloor n/2\rfloor}
  \frac{\sqrt{n!}}
  {4^k k!\sqrt{(n-2k)!}}\,
  \frac{z^{n-2k}}{\sqrt{(n-2k)!}} \\
&=\sum_{\substack{0\le m\le n \\ m\equiv n \,(\mathrm{mod} 2) }}
  \sqrt{\dfrac{n!}{m!}}\dfrac{1}{4^{\frac{n-m}{2}} (\frac{n-m}{2})!} 
  \frac{z^m}{\sqrt{m!}}. \\
\end{aligned}
\]
Inspired by the operator $\exp\left(\frac14\partial_z^2\right)$, 
we give the following definition.
\begin{defi}[Gram kernel matrix]\label{defi:gram_kernel_matrix}
The infinite Gram kernel matrix 
$Z = (z_{mn})_{m,n \geq 0} \in \mathbb{R}^{\mathbb{N} \times \mathbb{N}}$ of the problem 
is an upper triangular matrix of which the entries are
\[
\left\{
\begin{aligned}
&  z_{mn} = \sqrt{\dfrac{n!}{m!}} 
   \dfrac{1}{4^{\frac{n-m}2}(\frac{n-m}2)!},
&& n-m \in 2\mathbb{N},  \\
&  z_{mn} = 0,
&& \text{otherwise}.\\
\end{aligned}
\right.
\]
\end{defi}
\begin{lemma}
The Gram kernel matrix $Z$ satisfies
\[
Z L^\top = L^\top Z = I,
\]
where $I$ denotes the infinite identity matrix. 
\end{lemma}
\begin{proof}
By the definitions of $Z$ and $L$,
\[
  (Z^\top L)_{nm}
= \sum_{k=m}^n z_{kn} l_{km}
= \sum_{k=m}^n (-1)^{\frac{m-k}{2}}z_{kn} z_{mk}.
\]
The summand is nonzero only when the factorials
appearing in the definitions of $z_{kn}$ and $z_{mk}$ are well-defined.
Thus, the sum is nonzero only if $n \geq m$ and $n - m$ is even.
In that case,
\[
  (Z^\top L)_{nm}
= \sum_{k=m}^{n} (-1)^{\frac{m-k}{2}}
  \sqrt{ \dfrac{n!}{m!} }
  \dfrac{1}{4^{\frac{n-m}{2}} (\frac{n-k}{2})! (\frac{k-m}{2})!}.
\]
Introducing
$ r = \frac{n-m}{2} $,
$ j = \frac{n-k}{2} $,
we have
$ k = n - 2j $, and
$ j = 0,1,\ldots,r$.
Consequently,
\[
  (Z^\top L)_{nm}
= \sqrt{\frac{n!}{m!}}\,
  \frac{1}{4^r}
  \sum_{j=0}^{r} \frac{(-1)^{j-r}}{j!(r-j)!}.
\]
Thus,
for \( n > m \), the sum vanishes.
For \( n = m \), the sum givs \(1\).
\end{proof}
\begin{lemma}[Decomposition of the inverse Gram matrix]
\label{lem:decomposition of the inverse Gram matrix}
The inverse Gram matrix $A^{-1}$ admits a decomposition of the form
\[
A^{-1} = Z D^{-1} Z^\top,
\]
where the matrix $Z$, $D$ are given by~Definition \ref{defi:gram_kernel_matrix} 
and Equation~\eqref{eq:entries_cholesky_gram_D}.
\end{lemma}
\begin{proof}
By Proposition~\ref{prop:Cholesky_decomposition_of_the_Gram_matrix}, we have
\[
A^{-1} = (L D L^\top)^{-1} = (L^\top)^{-1} D^{-1} L^{-1} = Z D^{-1} Z^\top.
\]
\end{proof}
The identity above is identitical to quadratic forms on 
finitely supported sequences. 
They should not be interpreted as bounded operator identity on $\ell^2$. 
In fact, $D^{-1}$ grows exponentially, which is not bounded.

\section{Finite Gram matrix}
\label{sec:finite-gram-matrix}

Let $N\in\mathbb{N}$ be fixed. 
We consider the first $N+1$ scaled AW Hermite functions 
\[ \psi_0,\psi_1,\ldots,\psi_N, \] 
and denote by $A_{11}^N$ the corresponding finite Gram matrix. 
To relate this finite matrix to the infinite Gram matrix, 
we decompose the infinite Gram matrix $A$ according to the truncation index $N$ as:
$$
\renewcommand{\arraystretch}{1.5}
A=
\left(
\begin{array}{cc}
A_{11}^N & A_{12}^N \\
A_{21}^N & A_{22}^N \\
\end{array}
\right)
$$
where the blocks are
$$
\renewcommand{\arraystretch}{1.5}
\left\{
\begin{array}{clcl}
A_{11}^N &\in \mathbb R^{(N+1)     \times (N+1)    }, &
A_{12}^N &\in \mathbb R^{(N+1)     \times \mathbb N}, \\
A_{21}^N &\in \mathbb R^{\mathbb N \times (N+1)    }, &
A_{22}^N &\in \mathbb R^{\mathbb N \times \mathbb N}.
\end{array}
\right.
$$
Unless stated otherwise, all block matrices $A_{11}, A_{12}, A_{21}, A_{22}$ 
are understood to correspond to the truncation index $N$, 
and the superscript ${}^N$ is omitted whenever there is no ambiguity. 
We decompose the infinite matrices $Z$, $L$, and $D$ in the same manner.
\begin{lemma}
Let $N\in\mathbb{N}$ be fixed. 
Then
\begin{equation} \label{eq:aNN}
  a_{NN}
= 2^{-2N - \frac{1}2}
  \frac{ (2N)!} { \left( N ! \right)^2 }.
\end{equation}
For large $N$, we have 
$
a_{NN}\approx (\pi 2 N) ^{-\frac12} $. 
\end{lemma}
\begin{proof}
The Stirling formula written as  $N!\approx \sqrt{2\pi N} (N/e)^N$ yields that
$$
a_{NN} 
\approx 
2^{-2N - \frac{1}2} 
\frac{\sqrt{2\pi 2 N} (2N/e)^{2N}}{\left( \sqrt{2\pi N} (N/e)^N \right)^2}
\approx (\pi 2 N) ^{-\frac12}  .
$$
\end{proof}
This formula is consistent with the fact that the amplitude of 
the scaled AW Hermite functions decreases like $O(N^{-\frac14})$ 
in the main "support" of Hermite functions \cite{szego}.
It also indicates that the ratio of 
large numbers in (\ref{eq:aNN}) is asymptotically small.

\begin{lemma}\label{lem:relation_interesting}
Block matrices of $A$ and $Z$ satisfy
\[
  A_{11} Z_{12}
+ A_{12} Z_{22}
= 0.
\]
\end{lemma}
\begin{proof}
We have the Cholesky decomposition of $A$ 
\[
A = L D L^\top
\]
as given by Proposition~\ref{prop:Cholesky_decomposition_of_the_Gram_matrix}.
Multiplying $A$ by $Z$
we have
\[
A Z = L D L^{\top} Z = L D.
\]
Since $L D$ is a lower triangular matrix, which completes the proof. 
\end{proof}
\begin{Remark} 
The above identity motivates the terminology Gram kernel matrix. 
Indeed, every column of 
\[
\renewcommand{\arraystretch}{1.5}
\begin{bmatrix} 
Z_{12} \\ 
Z_{22} 
\end{bmatrix} 
\] 
belongs to the kernel of the block Gram matrix 
\[
\renewcommand{\arraystretch}{1.5}
\begin{bmatrix} 
A_{11} & A_{12} 
\end{bmatrix}. 
\] 
In other words, 
\[ 
\renewcommand{\arraystretch}{1.5}
\operatorname{Im} 
\begin{bmatrix} 
Z_{12} \\ 
Z_{22} 
\end{bmatrix} 
\subseteq 
\ker 
\begin{bmatrix} 
A_{11} & A_{12} 
\end{bmatrix}. 
\] 
\end{Remark}

\begin{lemma}\label{lemma:equation_some_2}
For a fixed integer $N$, 
block matrices of $A$, $Z$ and $L$ satisfy 
\[
  A_{11}^{-1}A_{12}
= Z_{11}L^\top_{21}
=-Z_{12}L^\top_{22}.
\]
\end{lemma}
\begin{proof}
Since the Gram kernel matrix $Z$ and 
the matrix $L^\top$ are upper triangle matrices,
we have
\[
Z_{12}L^\top_{22} = -Z_{11}L^\top_{21}.
\]
By Lemma \ref{lem:relation_interesting}, 
we have
\[
  A_{11} Z_{12}
+ A_{12} Z_{22}
= 0.
\]
Take the following product
\[
  \left(
  A_{11} Z_{12}
+ A_{12} Z_{22} \right) L^\top_{22}
= 0,
\]
which leads to
\[
A_{11}^{-1}A_{12} = -Z_{12}L^\top_{22}.
\]
\end{proof}

\begin{proposition}\label{prop:maximum_absolute_entry}
Consider the block matrix by truncation $N$ of the Gram matrix $A$.
The values of the entries of $A_{21} A_{11}^{-1} A_{12}$ are estimated as follows:
\[
  \left| \left(A_{21} A_{11}^{-1} A_{12}\right)_{nm} \right|
\leq \left| a_{nm} \right|.
\]
The upper bound for the entries is 
\[
  \max_{n,m}\left| \left(A_{21} A_{11}^{-1} A_{12}\right)_{nm} \right|
\leq \left| a_{N+1,N+1} \right|.
\]
\end{proposition}
\begin{proof}
Consider the Cholesky decomposition of the Gram matrix $A$, 
and let $Y$ denote the product of $L$ and $D$, 
we have
\[
\renewcommand{\arraystretch}{1.5}
Y L^\top = 
\begin{bmatrix}
Y_{11} L^\top_{11} & Y_{11}L^\top_{21} \\
Y_{21} L^\top_{11} & Y_{21}L^\top_{21} +Y_{22}L^\top_{22} \\
\end{bmatrix}
=
\begin{bmatrix}
A_{11} & A_{12} \\
A_{21} & A_{22} \\
\end{bmatrix}
.
\]
Then, we have
\[
\begin{aligned}
  \left( Y_{21}L^\top_{21} \right)_{nm}
=& \, a_{nm}
  \frac{(\frac{n+m}{2})! n!m!}{(n+m)!} 
  \sum_{k=0}^{N} \frac{2^k}{k!(\frac{n-k}{2})!(\frac{m-k}{2})!} \\
=& \, a_{nm} C_1(N,n,m), \\
\end{aligned}
\]
and
\[
\begin{aligned}
  \left( Y_{22}L^\top_{22} \right)_{nm}
=& \, a_{nm}
  \frac{(\frac{n+m}{2})! n!m!}{(n+m)!} 
  \sum_{k=N+1}^{\min(n,m)} \frac{2^k}{k!(\frac{n-k}{2})!(\frac{m-k}{2})!} \\ 
=& \, a_{nm} C_2(N,n,m). \\
\end{aligned}
\]
Clearly, $C_1, C_2$ are two non-negative numbers and $C_1 + C_2 = 1$
given by Lemma~\ref{eq:identity_hypergeometric_series}.
Hence, we get
\[
\left|\left( Y_{21}L^\top_{21} \right)_{nm}\right|
\leq |a_{nm}|.
\]
By Lemma \ref{lemma:equation_some_2}, we have
\[
\begin{aligned}
   \left|\left( Y_{21}L^\top_{21} \right)_{nm}\right|
&= \left|\left( A_{21}Z_{11}L^\top_{21} \right)_{nm}\right| \\
&= \left|\left( A_{21}A_{11}^{-1}A_{12} \right)_{nm}\right| \\
&\leq \left| a_{n,m} \right|.
\end{aligned}
\]
Moreover, the Gram matrix $A$ has diagonal entries 
that are the maximum (in absolute value) entries of their respective rows, and
the absolute value of the entries at the diagonal is nonincreasing as the row index increases.
Consequently, $\forall n, m \geq N+1$, 
\[
\left|\left( A_{21}A_{11}^{-1}A_{12} \right)_{nm}\right| 
\leq |a_{N+1,N+1}|,
\]
which completes the proof.
\end{proof}

\section{Decay properties of the finite Gram matrix}
\label{sec:decay-properties-of-the-finite-gram-matrix}

We denote the smallest eigenvalue of $A_{11}$ by $\lambda_N$.
To estimate the smallest eigenvalue $\lambda_N$ of the Gram matrix $A_{11}$, 
we can use the Rayleigh quotient
\[
  \lambda_N 
= \min\left\{ 
  \dfrac{ \sum_{n,m=0}^N a_{nm} x_n \overline{x}_m }{ \sum_{n=0}^N |x_n|^2 } 
  \right\},
\]
with coefficients $x_n$, $n=0,\dots,N$.
The Gram matrix $A_{11}$ is symmetric positive definite.
Since $A_{11}$ is symmetric positive definite, 
its inverse $A_{11}^{-1}$ is also symmetric positive definite. 
Consequently, 
all eigenvalues of $A_{11}^{-1}$ are positive, 
and its largest eigenvalue satisfies
\begin{equation}\label{eq:largest_eigenvalue}
  \dfrac{1}{\lambda_N} 
= \max\left\{ \sum_{n,m=0}^N (A_{11}^{-1})_{nm} x_n \overline{x}_m : 
  \sum_{n=0}^N |x_n|^2 = 1 \right\},
\end{equation}
where \(x=(x_n)_{n=0}^N \in\mathbb{C}^{N+1}\) is a unit vector.
Applying the Cauchy-Schwarz inequality to the entries of \(A_{11}^{-1}\),
\[
(A_{11}^{-1})_{nm} \leq \sqrt{ (A_{11}^{-1})_{nn} (A_{11}^{-1})_{mm} }
\]
for all \(0\le n,m\le N\),
and then using the Cauchy--Schwarz inequality we obtain an upper bound of
\[
\begin{aligned}
      \sum_{n,m=0}^N (A_{11}^{-1})_{nm} x_n \overline{x}_m
&\leq \sum_{n,m=0}^N \sqrt{ (A_{11}^{-1})_{nn} (A_{11}^{-1})_{mm} } 
      x_n \overline{x}_m \\
&\leq \left( \sum_{n=0}^N (A_{11}^{-1})_{nn} \right) 
      \left( \sum_{n=0}^N |x_n|^2 \right)  \\
&=    \left( \sum_{n=0}^N (A_{11}^{-1})_{nn} \right). \\
\end{aligned}
\]
Therefore, an upper bound for the largest eigenvalue of $A_{11}^{-1}$ 
immediately yields a lower bound for the smallest eigenvalue of $A_{11}$, that is
\begin{equation}\label{eq:lower_bound_smallest_eigenvalue}
\dfrac{1}{ \sum_{n=0}^N (A_{11}^{-1})_{nn} } \leq \lambda_N.
\end{equation}
In this section, following the approach of \cite{szego1936some,chen1999small}, 
we show that, by an appropriate choice of the vector \({x_n}\), 
the lower bound in~\eqref{eq:lower_bound_smallest_eigenvalue} is, in fact, an asymptotic estimate for large \(N\).

\subsection{Smallest eigenvalue}

\begin{lemma}[Asymptotics of the smallest eigenvalue]
\label{lem:asymptotics-of-the-smallest-eigenvalue}
Let \(\lambda_N\) denote the smallest eigenvalue of \(A_{11}\). 
Then, as \(N\to\infty\),
\[
     \frac{1}{\lambda_N}
\sim \sum_{n=0}^{N} (A_{11}^{-1})_{nn}
\sim \frac32 \sqrt{\frac{3}{\pi}} N^{-\frac12} 3^N.
\]
Consequently,
\[
  \lambda_N
\sim
  \frac{2\sqrt{3\pi}}{9} \,N^{\frac12}3^{-N}.
\]
In particular,
the smallest eigenvalue decays with exponential rate \(3^{-N}\).
\end{lemma}
\begin{proof}
From Lemma~\ref{lem:technical_lemma_2}, 
when $n, m$ are sufficiently large and satisfy~\eqref{eq:bulk_region}, we see that  
\begin{equation}\label{eq:decay_property_1}
|(A_{11}^{-1})_{nm}|
\sim
\sqrt{|(A_{11}^{-1})_{nn}| \, |(A_{11}^{-1})_{mm}|}.
\end{equation}
By choosing the vector $\{x_n\}$, as in \cite{szego1936some,chen1999small}, 
such that
\[
x_n = \left\{
\begin{aligned}
& \sigma e^{\imath\pi n} \sqrt{|(A_{11}^{-1})_{nn}|}, 
  \quad N - \alpha N^{1/4} \le n \le N, \\
& 0, \quad n < N - \alpha N^{1/4}, \\
\end{aligned}
\right.
\]
where $\sigma$ is a positive number determined by the condition 
\begin{equation}\label{eq:decay_property_2}
  \sum_{n=0}^N |x_n|^2 
= \sigma^2 \sum_{n = N - \alpha N^{1/4}}^N
  (A_{11}^{-1})_{nn} 
= 1.
\end{equation}
We find, using~\eqref{eq:decay_property_1} and~\eqref{eq:decay_property_2}, that
\[
\begin{aligned}
  \sum_{n,m=0}^N (A_{11}^{-1})_{nm} x_n \bar{x}_m 
&=\sum_{n,m = N - \alpha N^{1/4}}^N 
  \sigma^2 e^{\imath\pi (n-m)} (A_{11}^{-1})_{nm} 
  \sqrt{|(A_{11}^{-1})_{nn}| \, |(A_{11}^{-1})_{mm}|} \\
&\sim\sigma^2 \left( 
  \sum_{n = N - \alpha N^{1/4}}^N 
  (A_{11}^{-1})_{nn} 
  \right)^2 \\
&=\sum_{n = N - \alpha N^{1/4}}^N 
  (A_{11}^{-1})_{nn} 
\end{aligned}
\]
Recalling~\eqref{eq:largest_eigenvalue}, 
we see that the asymptotic behaviour of the maximum, 
by virtue of the inequality~\eqref{eq:lower_bound_smallest_eigenvalue}, 
is well approximated by
$\sum_{n=0}^N (A_{11}^{-1})_{nn}$. 
Therefore we have shown that
\[
\dfrac{1}{\lambda_N} \sim \sum_{n=0}^N (A_{11}^{-1})_{nn}.
\]
From Lemma~\ref{lem:technical_lemma_1}, we see that
\[
\dfrac{1}{\lambda_N} 
\sim \frac32 \sqrt{\frac{3}{\pi}} N^{-\frac12} 3^N.
\]
Inverting this asymptotic relation 
yields the claimed asymptotic formula for $\lambda_N$, 
and the proof is complete.
\end{proof}

\subsection{Schur complement}
\begin{lemma}\label{lem:schur-complement-A11}
Let $N \ge 1$.
Suppose that the block $A_{11}^N$ of the Gram matrix $A$ is partitioned as
\[
\renewcommand{\arraystretch}{1.5}
A_{11}^N =
\begin{pmatrix}
A_{11}^{N-1} & X   \\
X^\top       & a_{NN}
\end{pmatrix},
\]
where
\[
A_{11}^{N-1} \in \mathbb{R}^{N\times N}, \qquad
X            \in \mathbb{R}^{N\times 1}, \qquad
a_{NN}       \in \mathbb{R}.
\]
Since $A_{11}^{N-1}$ is invertible, 
the Schur complement of $A_{11}^{N-1}$ in $A_{11}^N$ is given by
\[
   A_{11}^N / A_{11}^{N-1}
:= a_{NN} - X^\top (A_{11}^{N-1})^{-1} X.
\]
Moreover, this Schur complement is given explicitly by
\[
  A_{11}^N / A_{11}^{N-1}
= 2^{-N-\frac12}.
\]
\end{lemma}
\begin{proof}
We have
\[
\begin{aligned}
A_{11}^{N} / A_{11}^{N-1} 
&= a_{NN} - X^\top (A_{11}^{N-1})^{-1} X \\
&= a_{NN} - \left(A_{21}^{N-1} (A_{11}^{N-1})^{-1} A_{12}^{N-1} \right)_{NN}. \\
\end{aligned}
\]
Since, for $n,m \ge N$ such that $n+m$ is even,
\[
\begin{aligned}
  \left(A_{21}^{N} (A_{11}^{N})^{-1} A_{12}^{N} \right)_{nm}
=& \, a_{nm}
  \frac{(\frac{n+m}{2})! n!m!}{(n+m)!} 
  \sum_{\substack{k=0\\ k\equiv n\,(\mathrm{mod}\,2)}}^{N}
  \frac{2^k}{k!(\frac{n-k}{2})!(\frac{m-k}{2})!}. \\
\end{aligned}
\]
Given by Lemma~\ref{eq:identity_hypergeometric_series}, we have
\[
\begin{aligned}
  a_{nm} - \left(A_{21}^{N} (A_{11}^{N})^{-1} A_{12}^{N} \right)_{nm}
=& \, a_{nm}
  \frac{(\frac{n+m}{2})! n!m!}{(n+m)!} 
  \sum_{\substack{k=N+1\\ k\equiv n\,(\mathrm{mod}\,2)}}^{\min(m,n)}
  \frac{2^k}{k!(\frac{n-k}{2})!(\frac{m-k}{2})!} \\ 
\end{aligned}
\]
Therefore
\[
\begin{aligned}
A_{11}^{N} / A_{11}^{N-1} 
&= a_{NN}
  \frac{(\frac{N+N}{2})! N!N!}{(N+N)!}
  \frac{2^N}{N!(\frac{N-N}{2})!(\frac{N-N}{2})!} \\
&= a_{NN} \dfrac{N!N!2^N}{(2N)!}
\end{aligned}
\]
Since
\[
a_{NN} = 2^{-2N-\frac12} \dfrac{(2N)!}{N!N!},
\]
as given by Lemma~\ref{eq:aNN},
we have
\[
  A_{11}^{N} / A_{11}^{N-1}
= 2^{-N-\frac12}.
\]
\end{proof}

\section{Conclusion}
\label{sec:conclusion}
In this work, we focus on the Gram matrix 
which arises from standard inner products of 
AW Hermite functions. 
We derive explicit expressions for the coefficents of the infinite Gram matrix, 
and its Cholesky decomposition.
We also show that the factorization of the infinite Gram matrix
can be interpreted on a scaled Bargmann–Fock basis.
Then we study the finite Gram matrix 
and show the decay properties of the finite Gram matrix.
We study the smallest eigenvalue $\lambda_N$ of the finite Gram matrix 
and prove that $\lambda_N$ has exponential decay to zero.
We also investigate the exponential decay of 
the Schur complement of the finite Gram matrix.
These properties may be used in the numerical analysis of 
the Galerkin spectral method based on the AW Hermite functions 
applied to the VP system. 

In addition to the direct applications to numerical analysis described above, 
the results obtained for the Gram matrix associated with the standard inner product of the AW Hermite functions 
may further enrich the theory of orthogonal polynomials and Gram matrices. 
They may also have broader implications for their many applications in approximation theory.

\appendix

\section{Technical lemmas}

\subsection{Trace estimate}

\begin{lemma}\label{lem:technical_lemma_1}
For a large integer $N$,
we have
\[
  tr(A_{11}^{-1}) 
\sim 
  \frac32 \sqrt{\frac{3}{\pi}} N^{-\frac12} 3^N.
\]
\end{lemma}
\begin{proof}
We consider the case that $N$ is even, i.e. $N = 2M$ with $M\in\mathbb{N}$.
Since
\[
  \operatorname{tr} \left( A_{11}^{-1}  \right)
= \operatorname{tr} \left( Z_{11} D_{11}^{-1} Z_{11}^\top \right)
= \operatorname{tr} \left( D_{11}^{-1} Z_{11}^\top Z_{11} \right),
\]
it follows that
\[
\begin{aligned}
  \operatorname{tr} \left( A_{11}^{-1} \right)
&= \sum_{i=0}^N     2^{i+\frac12}    \sum_{j=0}^N z_{ji}^2 \\
&= \sum_{m=0}^{M-1} 2^{2m+1+\frac12} \sum_{r=0}^{M-1} z_{2r+1,2m+1}^2
 + \sum_{m=0}^M     2^{2m+\frac12}    \sum_{r=0}^M z_{2r,2m}^2. 
\end{aligned}
\]
We now consider the contribution from the even-indexed diagonal entries, 
corresponding to the second sum above. For each $r=0,\ldots,M$,
\[
\begin{aligned}
  \sum_{r=0}^M z_{2r, 2m}^2
&=\sum_{r=0}^M \dfrac{(2m)!}{(2r)!16^{m-r}((m-r)!)^2} \\
&=\sum_{r=0}^m \dfrac{(2m)!}{(2r)!16^{m-r}((m-r)!)^2} \\
&=\sum_{k=0}^m \dfrac{(2m)!}{(2m-2k)!16^{k}(k!)^2}.
\end{aligned}
\]
The third equality holds since $z_{2r,2m} = 0$ for all $r > m$. 
The last equality follows by setting $k = m - r$.
Let $k = x m$, $0 < x < 1$  and define the following function 
\[
b_m(x) := \dfrac{(2m)!}{(2m(1-x))!16^{xm}((xm)!)^2},
\]
and using Stirling's formula
\[
\ln(m!) = m\ln m - m + \dfrac{1}{2} \ln (2\pi m) + \mathcal{O}(m^{-1}),
\]
we obtain
\[
  b_m(x)
= p_2(x) \exp( m \,p_1(x) ),
\]
where 
\[
p_1(x) = -2(1-x)\ln(1-x) - 2x\ln (2x)
\]
and 
\[
  p_2(x) 
= \dfrac{1}{2\pi m x\sqrt{1-x}} + 
  \exp\left(
  \mathcal{O}\left(
  \dfrac{1}{m}\left(
  \dfrac{1}{2}
- \dfrac{1}{2(1-x)}
- \dfrac{1}{x}\right)\right)\right).
\]
The leading behaviour of this sum for large m is 
in turn found by replacing the sum by an integral 
and by applying interior Laplace's method, 
which in this context may be stated as :
if for $x \in (a,b)$, the real continuous function $g(x)$ has 
its maximum the value $g(x_0)$ at an interior point $x_0$, then as $m\to\infty$
\[
  \int_a^b f(x) e^{m g(x)} dx 
\sim f(x_0) e^{m g(x_0)}
  \sqrt{\dfrac{2\pi}{m|g^{\prime\prime}(x_0)|}}.
\]
Applying the method gives
\[
\begin{aligned}
      \sum_{k=0}^m b_m(k/m) 
\sim  m \int_0^1 b_m(x) dx 
&\sim m \, p_2(x_0) \exp(m p_1(x_0))
      \sqrt{\dfrac{2\pi}{m|p_1^{\prime\prime}(x_0)|}} \\
&=    \left(\dfrac{3}{4}\right)^{\frac12} \pi^{-\frac12}
      m^{-\frac12}\left(\dfrac{9}{4}\right)^m.
\end{aligned}
\]
We can also prove for the case where the contribution from the odd-indexed diagonal entries. 
Thus we obtain, for $i$ is very large,
\[
  \left( A_{11}^{-1} \right)_{ii}
= 2^{i+\frac12} \sum_{j=0}^N z_{ji}^2
\sim 
  \sqrt{\frac{3}{\pi}} i^{-\frac12} 3^i.
\]
Collecting all entries on the diagonal gives
\[
  tr(A_{11}^{-1}) 
\sim 
  \frac32 \sqrt{\frac{3}{\pi}} N^{-\frac12} 3^N.
\]
\end{proof}

\subsection{Asymptotic behavior in the bulk region}

\begin{lemma}\label{lem:technical_lemma_2}
Fix a positive constant $\alpha$. 
Then, for a large integer $N$ and indices $i, j$ in the bulk region
\begin{equation}\label{eq:bulk_region}
N - \alpha N^{1/4} \le i,j \le N,
\end{equation}
such that $N-i$ and $N-j$ are even,
we have
\[
(A_{11}^{-1})_{ij}
\sim
\sqrt{(A_{11}^{-1})_{ii}\,(A_{11}^{-1})_{jj}}.
\]
Equivalently,
\[
(A_{11}^{-1})_{ij} \sim 2^{N+\frac12} z_{iN} z_{jN}, \qquad
(A_{11}^{-1})_{ii} \sim 2^{N+\frac12} z_{iN}^2, \qquad
(A_{11}^{-1})_{jj} \sim 2^{N+\frac12} z_{jN}^2,
\]
where $(z_{iN})$ are entries of the Gram kernel matrix.
\end{lemma}
\begin{proof}
Suppose
\[
N-\alpha N^{1/4}\le i,j\le N,
\]
where \(\alpha>0\) is fixed. 
Equivalently,
\(i = N-a, j = N-b,\)
with
\(
0\le a,b\le \alpha N^{1/4},
\)
that is,
\(
a,b=\mathcal{O}(N^{1/4}).
\)
Writing
\(
k=N-2r,
\)
the sum becomes
\[
  (A_{11}^{-1})_{ij}
= \sum_{\substack{k=\mathrm{max}(i,j) \\ 
	          i\equiv k\,(\mathrm{mod}\,2)}}^{N}
  \dfrac{z_{ik} z_{jk}}{d_k}
= \sum_{r=0}^{\mathrm{min}(\frac{a}2, \frac{b}2)} w_r^{(ij)},
\]
where
\[
w_r^{(ij)}
=
2^{N-2r+\frac12}
\frac{(N-2r)!}
{\sqrt{i!\,j!}}
\frac{1}
{4^{\frac{a-2r}{2}}
\left(\frac{a-2r}{2}\right)!
\,4^{\frac{b-2r}{2}}
\left(\frac{b-2r}{2}\right)!},
\]
Since $ a, b =  \mathcal{O}(N^{\frac14}) $,
there are only $\mathcal{O}(N^{\frac14}) $ admissible values of \(r\).
Since, using Stirling's formula,
\[
  \frac{(N-2r)!}{N!}
= N^{-2r} \left( 1+\mathcal{O}(N^{-\frac12}) \right),
\]
and similarly for $\frac{N!}{i!}$ and $\frac{N!}{j!}$,
we obtain
\[
  \frac{(N-2r)!}{\sqrt{i!\,j!}}
= N^{\frac{a+b}{2}-2r}
  \left(1+\mathcal{O}(N^{-\frac12})\right).
\]
Hence
\[
  w_r^{(ij)}
= 2^{N+\frac12} N^{\frac{a+b}{2}} C_r(a,b) N^{-2r}
  \left(1+\mathcal{O}(N^{-\frac12})\right),
\]
where
\[
  C_r(a,b)
= \frac{2^{-2r}}
  {4^{a/2-r}\, 4^{b/2-r}\,
  \left(\frac{a-2r}{2}\right)!
  \left(\frac{b-2r}{2}\right)!}.
\]
Since
\[
  \frac{(\frac{a}2)!}{(\frac{a-2r}2)!}
= \mathcal{O} \left(a^r\right)
= \mathcal{O} \left(N^{r/4}\right),
\]
and similarly for \(b\), it follows that
\[
  \frac{w_r^{(ij)}}{w_0^{(ij)}}
= \mathcal{O} \left(N^{-3r/2}\right).
\]
Consequently,
\[
(A_{11}^{-1})_{ij}
=
w_0^{(ij)}
\left(1+\mathcal{O}(N^{-3/2})\right).
\]
The leading contribution is
\[
  (A_{11}^{-1})_{ij}
\sim 
  2^{N+\frac12} \frac{N^{(a+b)/2}} {4^{a/2}4^{b/2}
  (a/2)!(b/2)!}
  \left(1+\mathcal{O}(N^{-\frac12})\right).
\]
Thus,
\[
(A_{11}^{-1})_{ij}
\sim
\sqrt{(A_{11}^{-1})_{ii}(A_{11}^{-1})_{jj}},
\qquad
N-\alpha N^{1/4}\le i,j\le N.
\]
Equivalently,
\[
(A_{11}^{-1})_{ij} \sim 2^{N+\frac12} z_{iN} z_{jN}, \qquad
(A_{11}^{-1})_{ii} \sim 2^{N+\frac12} z_{iN}^2, \qquad
(A_{11}^{-1})_{jj} \sim 2^{N+\frac12} z_{jN}^2,
\]
showing that the last contribution \(k=N\) 
asymptotically determines all three quantities. 
\end{proof}

\section{Application to the VP system}
\label{sec:application-to-the-vp-system}
We briefly describe the Galerkin spectral method based on AW Hermite functions 
to the one-dimensional VP system. 
Let $f$ be the solution of the VP system with spectral expansion
\[
  f(t,x,v)   
= \sum_{n=0}^{\infty} \widehat{u}_n(t,x) \psi_n(v),
\]
where $(\widehat{u}_n(t,x))_{n\ge 0}$ are 
Hermite coefficients~\cite{fok2002combined}.
The solution $f$ satisfies
\begin{equation}\label{eq:vp-system}
\left\{
\begin{aligned}
&\partial_t f + v\partial_x f + E(t,x) \partial_v f = 0, \\
&\partial_x E(t,x) = \rho - \rho_0,
\end{aligned}
\right.
\qquad
\rho(t,x) = \int_{\mathbb R} f(t,x,v)\,dv,
\end{equation}
with periodic boundary conditions in 
$x\in \Omega_x = ]0,x_{\text{max}}]$ and $v\in\mathbb R$. 
$E(t, x)$ is the electrostatic field, and $\rho(tx)$ is the density.
The constant $\rho_0$ is chosen to satisfy the quasi-neutrality condition
\[
\int_{\Omega_x} (\rho-\rho_0)\,dx=0.
\]
This formulation follows the model considered in the reference work 
\cite{armstrong1966numerical,filbet2003comparison,schumer1998vlasov,gagne1977splitting,kormann2021generalized,bessemoulin2022stability,bessemoulin2023convergence,manzini2017convergence,funaro2021stability,pagliantini2023physics}.

The Galerkin method is obtained by using the same AW Hermite functions as trial and test functions
\begin{equation}\label{eq:hermite-expansion}
  f_N(t,x,v)
= \sum_{m=0}^N u_m(t,x) \psi_m(v).
\end{equation}
Inserting the expansion~\eqref{eq:hermite-expansion} 
into the VP system~\eqref{eq:vp-system}, 
using the recurrence relations~\eqref{recursive_relation_v} 
and~\eqref{recursive_relation_partial_v},
and testing the Vlasov equation with $\psi_n$, 
the semi-discrete system reads
\begin{equation}\label{eq:the-semi-discrete-system}
\begin{aligned}
  \sum_{m=0}^N a_{nm}\partial_tu_m
&+\sqrt{\frac{T}{2}} \sum_{m=0}^N \left( 
  \sqrt{m}   a_{n,m-1}
 +\sqrt{m+1} a_{n,m+1}
  \right) \partial_x u_m \\
&-\sqrt{\frac{2}{T}} \,E_N
  \sum_{m=0}^N \sqrt{m+1} a_{n,m+1}u_m
 =0, \qquad n=0,\ldots,N,
\end{aligned}
\end{equation}
where $E_N$ is coupled to the approximation of the density 
through the Poisson equation. 
In contrast to the Petrov--Galerkin method, the AW Hermite functions 
are not orthogonal in the standard $L^2(\mathbb R)$ inner product. 
Consequently, the Galerkin method introduces 
the blocks $A_{11}$ and $A_{12}$ of the Gram matrix $A$.
We show this in the following.
Introducing
\[
U_1= (u_0,\ldots,u_N)^\top,
\]
the semi-discrete system~\eqref{eq:the-semi-discrete-system} 
can be expressed equivalently as
\begin{equation}\label{eq:the-semi-discrete-system-matrix} 
\begin{aligned}
  A_{11} \partial_t U_1
&+\sqrt{\frac{T}{2}}     (A_{11} B_{11} + A_{12} B_{21}) \partial_x U_1 \\
&-\sqrt{\frac{2}{T}} E_N (A_{11} D_{11} + A_{12} D_{21}) U_1
= 0,
\end{aligned}
\end{equation}
together with the discretized Poisson equation. 
In this appendix, $B$ represents the transport matrix 
and $D$ the acceleration matrix.
The infinite triangular and sparse matrices are given by
$B, D \in \mathbb{R}^{\mathbb{N}\times\mathbb{N}}$
\[
\begin{aligned}
&B = (b_{mn})_{m,n\geq 0}, \quad b_{mn} = \sqrt{m}\,\delta_{m-1,n} + \sqrt{m+1}\,\delta_{m+1,n}, \\
&D = (d_{mn})_{m,n\geq 0}, \quad d_{mn} = \sqrt{m}\,\delta_{m-1,n}.
\end{aligned}
\]
$B_{11}$, $B_{21}$ represents the blocks of $B$ and 
$D_{11}$, $D_{21}$ the blocks of $D$.

This formulation provides the Galerkin spectral method 
to the VP system in velocity space. 
The particular treatment of the resulting semi-discrete system, 
including the handling of the Gram matrix 
and the analysis of the conservation properties, 
is presented in the following.

\subsection{Sparse structure}
Since the Gram matrix arising from the Galerkin method is dense, 
its direct use can incur substantial computational costs. 
To address this issue, we derive an equivalent formulation 
that preserves the sparsity structure.
We can equivalently write 
the semi-discrete system~\eqref{eq:the-semi-discrete-system-matrix} as
\begin{equation}\label{eq:the-semi-discrete-system-matrix-A11}
\begin{aligned}
   \partial_t U_1
+ &\sqrt{\dfrac{T}{2}}
   (B_{11} + A_{11}^{-1}A_{12}B_{21}) \partial_x U_1 \\
- &\sqrt{\dfrac{2}{T}} E_N
   (D_{11} + A_{11}^{-1}A_{12}D_{21}) U_1
= 0.
\end{aligned}
\end{equation}
Using Lemma~\ref{lemma:equation_some_2}
in the system~\eqref{eq:the-semi-discrete-system-matrix-A11},
we obtain
\begin{equation}\label{eq:the-semi-discrete-system-matrix-Z12}
\begin{aligned}
   \partial_t U_1
+ &\sqrt{\dfrac{T}{2}}
   (B_{11} - Z_{12}L_{22}^{\top} B_{21}) \partial_x U_1 \\
- &\sqrt{\dfrac{2}{T}} E_N
   (D_{11} - Z_{12}L_{22}^{\top} D_{21}) U_1
= 0.
\end{aligned}
\end{equation}
Since $B_{21}$ (or $D_{21}$) has only one nonzero entry,
located in its first row, last column,
the matrix $B_{11} - Z_{12}L_{22}^{\top}B_{21}$
(or $D_{11} - Z_{12}L_{22}^{\top}D_{21}$)
differs from $B_{11}$ (or $D_{11}$) only in its last column.
Thus, the system~\eqref{eq:the-semi-discrete-system-matrix-Z12}
preserves the sparsity structure
of the classical Petrov–Galerkin method,
apart from the additional nonzero entries in the last column.
From a numerical standpoint,
the the system~\eqref{eq:the-semi-discrete-system-matrix-Z12} 
therefore avoids the direct use of the dense Gram matrix
while introducing only a negligible additional computational cost.

\subsection{Sharper estimate}
For spectral methods applied to the VP system, 
obtaining a sharp estimate for the highest moment $u_N$ is of particular interest. 
Assume that
\[
\begin{aligned}
  \left\| \sum_{n=0}^N \left( u_n - \widehat{u}_n \right) \psi_n 
  \right\|_{L^2(dvdx)}
&\le
  \|f_N-f\|_{L^2(dv\,dx)}
  +\left\| \sum_{n=N+1}^{\infty}\widehat{u}_n\psi_n \right\|_{L^2(dv\,dx)} \\
&\le g(N),
\end{aligned}
\]
where $g(N)\to0$ as $N\to\infty$.
Here, \(L^2(dv\,dx)\) denotes the space of measurable functions 
that are square-integrable with respect to 
the product measure \(dv\,dx\), 
equipped with the norm
$$
   \|f\|_{L^2(dv\,dx)}
:= \left( \int_{\Omega_x} \int_{\mathbb{R}}
   |f(v,x)|^2 \,dv\,dx \right)^{\frac12}.
$$
$\lambda_N$ is the smallest eigenvalue of the Gram matrix $A_{11}^N$ and 
its estimate is given by Lemma~\ref{lem:asymptotics-of-the-smallest-eigenvalue}.
We have 
\[
  \left\| \sum_{n=0}^N \left( u_n - \widehat{u}_n \right) \psi_n 
  \right\|_{L^2(dvdx)}
\ge 
  \lambda_N \int_{\Omega_x} ( u_N - \widehat{u}_N )^2 \,dx.
\]
Applying the Cauchy--Schwarz inequality,
\[
  \left( \int_{\Omega_x} |u_N-\widehat{u}_N| \,dx \right)^2
\le 
  |\Omega_x| \int_{\Omega_x}  (u_N-\widehat{u}_N)^2\,dx,
\]
yields
\[
  \left( \int_{\Omega_x} |u_N-\widehat{u}_N| \,dx \right)^2
  \le \frac{|\Omega_x| \,g^2(N)}{\lambda_N}.
\]
Alternatively, using the Schur complement, we have 
\[
\begin{aligned}
  \left\| \sum_{n=0}^N
  \left( u_N - \widehat{u}_N \right) \psi_n \right\|_{L^2(dvdx)}^2
&=
  \int_{\Omega_x}
  \left( U_1 - \widehat{U}_1 \right)^\top A_{11}^N
  \left( U_1 - \widehat{U}_1 \right) \,dx \\
&\geq
  \int_{\Omega_x} 
  A_{11}^{N} / A_{11}^{N-1} (u_N - \widehat{u}_N)^2 \,dx \\
&\geq 
  \dfrac{A_{11}^{N} / A_{11}^{N-1}}{|\Omega_x|} 
  \left( \int_{\Omega_x} |u_N - \widehat{u}_N| \,dx \right)^2.
\end{aligned}
\]
Consequently, as given by Lemma~\ref{lem:schur-complement-A11}, we have
\[
\begin{aligned}
  \left( \int_{\Omega_x} |u_N - \widehat{u}_N| \,dx \right)^2
&\leq 
  \dfrac{|\Omega_x| \,g^2(N)}{A_{11}^{N} / A_{11}^{N-1}}.
\end{aligned}
\]
Hence, the estimate obtained from the Schur complement 
is sharper than the one based on the smallest eigenvalue of the Gram matrix.

\end{document}